\documentclass[12pt]{amsart}
\ProvidesLanguage{portuges}
\usepackage[latin1]{inputenc}
\usepackage{indentfirst}
\usepackage{amsfonts}
\usepackage{graphicx}
\usepackage{amsmath}
\usepackage{amssymb}
\usepackage{latexsym}
\usepackage{amscd}
\usepackage[all]{xy}
\usepackage{cancel}
\usepackage{color}

\newcommand{\F}{\mathbb {F}}

\newtheorem{theorem}{Theorem}[section]
\newtheorem{definition}[theorem]{Definition}
\newtheorem{lemma}[theorem]{Lemma}

\newtheorem{proposition}[theorem]{Proposition}
\newtheorem{remark}[theorem]{Remark}
\newtheorem{example}[theorem]{Example}

\DeclareMathOperator{\Div}{Div}

\title{Subcovers and codes on the $X_{n,r}$ curves}

\author[borges]{H. Borges}

\author[castellanos]{A. S. Castellanos}

\author[tizziotti]{G. Tizziotti}

\begin{document}
\maketitle

\begin{abstract}
In this work, subcovers $\mathcal{X}_{n,r}^s$ of the curve $\mathcal{X}_{n,r}$ are constructed, the Weierstrass semigroup $H(P_\infty)$ at the point $P_\infty \in \mathcal{X}_{n,r}^s$  is determined and   the co\-rresponding one-point AG codes are investigated. Codes establishing new records on the parameters with respect to the previously known ones are discovered, and $108$ improvements on MinT tables are obtained.

\end{abstract}

%\textbf{Index terms:} Subcover, AG codes, Weierstrass semigroup, one-point AG codes, $\mathcal{X}_{n,r}$ curve.

\section{Introduction}

Let  $\mathbb{F}_q$ be the finite field with $q$ elements, where $q$ is a power of a prime $p$, and let
$\mathcal{X}$ be  a (projective, geometrically irreducible, algebraic) curve defined over $\mathbb{F}_q$.
%We will refer to $\mathcal{X}$ simply as curve.
%
%Let $\mathcal{X}$ be a nonsingular, projective, geometrically irreducible algebraic curve of positive genus $g=g(\mathcal{X})$ defined over a finite field $\mathbb{F}_q$ with $q$ elements. In this work we will refer to this simply as a \textit{curve}, and let $\mathbb{F}_q(\mathcal{X})$ be its associated function field.

In  \cite{c2},  Borges and Concei\c c\~ao  introduced the  curve $\mathcal{X}_{n,r}$, which is the  curve defined over $\mathbb{F}_{q^n}$  with   affine equation  given by
\[T_n(y)=f_r(x)\;,\]
where $n\geq 2$, $r\in \{\lceil \frac{n}{2}\rceil,\ldots, n-1\}$ with $\mbox{gcd}(n,r)=1$, $T_n(z)=z+z^q+\cdots +z^{q^{n-1}}$, and $f_r(x):=T_n(x^{1+q^r})\pmod{x^{q^n}-x}$.

%Note that for  $p>2$, this curve can be seen as a generalization of the Hermitian curve.

The curve   $\mathcal{X}_{n,r}$ has   interesting  arithmetic  and geometric  properties as  it is an  $\F_{q^n}$-Frobenius nonclassical   Castle curve,  with  a large number of  automorphisms (see \cite{c2}, \cite{bst}). It was observed in  \cite[Corollary 3.5]{bst} that  $\mathcal{X}_{n,r}$
has a simpler plane model   with   affine equation  given by
\begin{equation}\label{simplemodel}
y^{q^{n-1}}+\cdots+y^q+y=x^{q^{n-r}+1}-x^{q^{n}+q^{n-r}}.
\end{equation}
In this paper, for $s=1,2,\ldots,n-1$, we will construct subcovers of  \eqref{simplemodel} given by
\begin{equation}\label{covermodel}
\mathcal{X}_{n,r}^s: \  g_s(y)=x^{q^{n-r}+1}-x^{q^{n}+q^{n-r}}.
\end{equation}
where $g_s(y)=y^{q^s}+a_{s-1}y^{q^{s-1}}+\cdots+a_0y \in \F_{q^n}[x]$ are separable $q$-polynomials. Then   some properties  of $\mathcal{X}_{n,r}^s$, such as their  number of $\F_{q^n}$-rational points and genus, as well as   the Weierstrass semigroups $H(P_\infty)$ at their only  point at infinity $P_\infty$ are investigated.

 Weierstrass semigroup is a classical object in the theory of algebraic curves  that
  is related to both  theoretical and applied topics.   Its particular connection with  Algebraic Geometric (AG) codes, introduced by Goppa in   \cite{Goppa1}, \cite{Goppa2},  led to Weierstrass semigroups  being  extensively studied over past decades (see e.g. \cite{fanali}, \cite{gretchen}, \cite{gretchen2}, \cite{munuera} and \cite{ST}).
Our investigation of the semigroups $H(P_\infty)$ will lead  to  the construction of AG codes from the curve $\mathcal{X}_{n,r}$  and its subcovers, with  parameters  giving rise to  new records.

The  paper is organized as follows.  Section 2  presents  notations and results regarding   Weierstrass semigroups and AG codes. Section 3 is devoted to constructing subcovers of  curve $\mathcal{X}_{n,r}$ and to computing their genus, number of $\F_{q^n}$-rational points, and  the Weierstrass semigroup $H(P_{\infty})$ for a certain class of  subcovers $\mathcal{X}_{n,r}^{s}$.  Finally, in  Section 4,  we investigate AG codes arising from $\mathcal{X}_{n,r}$ and $\mathcal{X}_{n,r}^{s}$  and then codes attaining new records for  their parameters are presented.

%%%%%%%%%%%%%%%%%%%%%%%%%%%%%%%%%%%%%%%%%%%%%%%%%%%%%%%%%%
\begin{center}
\bf{Notation}
\end{center}

The following notation will be used throughout this text.

\begin{itemize}
\item $\mathbb{F}_q$ is the finite field with $q=p^m$ elements, where $p$ is a prime number.
\item $\mathcal{X}$ is  a curve defined over   $\mathbb{F}_q$.
\item $\# \mathcal{X}(\mathbb{F}_q)$ denotes the  number of $\mathbb{F}_q$-rational points, or simply rational  points, on $\mathcal{X}$.
\item $\mathbb{F}_q(\mathcal{X})$ is  the field of rational functions and $\mbox{Div}(\mathcal{X})$ is the set of divisors on $\mathcal{X}$.
\item For $f \in \mathbb{F}_q(\mathcal{X})$,  the zero and the pole divisor of $f$  are denoted by  $(f)_0$ and $(f)_{\infty}$, respectively.
\item For $D \in \mbox{Div}(\mathcal{X})$, the support of $D$ is denoted by  $\mbox{Supp}(D)$, and the  Riemann-Roch space associated to $D$  is  $\mathcal{L}(D):= \{ f \in \mathbb{F}_{q}(\mathcal{X}) \mbox{ : } (f) + D \geq 0 \} \cup \{ 0 \}$.
\item The  dimension of $\mathcal{L}(D)$ is  denoted by $\ell(D)$.
\item $\Omega(G)$ denotes  the space of differentials $\eta$ on $\mathcal{X}$ such that $\eta=0$ or $\mbox{div}(\eta)\geq G$, where $\mbox{div}(\eta)= \sum_{P\in \mathcal{X}}\mbox{ord}_{P}(\eta)P$ and $\mbox{ord}_P(\eta)$ is the order of $\eta$ at $P$.
\item  For $P$ on $\mathcal{X}$, $\mbox{v}_P$ is the discrete valuation at $P$, and for a differential $\eta$ on $\mathcal{X}$, $\mbox{res}_{P}(\eta)$ is the residue of $\eta$ at $P$.
\end{itemize}

%%%%%%%%%%%%%%%%%%%%%%%%%%%%%%%%%%%%%%%%%%%%%%%%%%%%%%%%%%

\section{Preliminaries}

%
%Throughout the paper we will use the following notations. Let $\mathcal{X}$ be a curve over $\mathbb{F}_q$, where $q$ is a power of the prime $p$. We will denote the number of $\mathbb{F}_q$-rational points - we will write simply rational points - of $\mathcal{X}$ by $\# \mathcal{X}(\mathbb{F}_q)$. Let $\mathbb{F}_q(\mathcal{X})$ be the field of rational functions and $Div(\mathcal{X})$ be the set of divisors on $\mathcal{X}$. For $f \in \mathbb{F}_q(\mathcal{X})$, we will denote the divisor and the divisor of poles of $f$ by $(f)$ and $(f)_{\infty}$, respectively. For $G \in Div(\mathcal{X})$, let $Supp(G)$ be the support of $G$ and let $\mathcal{L}(G):= \{ f \in \mathbb{F}_{q}(\mathcal{X}) \mbox{ ; } (f) + G \geq 0 \} \cup \{ 0 \}$ be the Riemann-Roch space associated to $G$. The dimension of $\mathcal{L}(G)$ as an $\mathbb{F}_{q}$-vector space will be denote by $\ell(G)$. Let $\Omega(G)$ be the space of differentials $\eta$ on $\mathcal{X}$ such that $\eta=0$ or $div(\eta)\geq G$, where $div(\eta)= \sum_{P\in \mathcal{X}}ord_{P}(\eta)P$ and $ord_P(\eta)$ is the order of $\eta$ at $P$. For $P$ on $\mathcal{X}$, $v_p$ is the discrete valuation at $P$ and for a differential $\eta$ on $\mathcal{X}$, $res_{P}(\eta)$ is the residue of $\eta$ at $P$.
%%As follows, we denote $\mathbb{N}_{0} = \mathbb{N} \cup \{0\}$, where $\mathbb{N}$ is the set of positive integers.

\subsection{Weierstrass Semigroup} Let $P$ be a rational point on $\mathcal{X}$ and $\mathbb{N}_{0}$ be the set of non-negative integers. The set
$$
H(P):= \{n \in \mathbb{N}_{0} \mbox{ : } \exists f \in \mathbb{F}_{q}(\mathcal{X}) \mbox{ with } (f)_{\infty} = n P \}
$$
is a numerical semigroup, called the \textit{Weierstrass semigroup} of $\mathcal{X}$ at $P$. The set $G(P) = \mathbb{N}_{0} \setminus H(P)$ is called  the  \textit{Weierstrass gap set} of $P$. By Weierstrass Gap Theorem, the cardinality of $G(P)$ is $g=g(\mathcal{X})$, the genus of the curve. In addition, $G(P)=\{\alpha_1,\ldots, \alpha_{g}\}$ with $1=\alpha_1 < \cdots < \alpha_{g} \leq 2g-1$. The semigroup $H(P)$ is called \textit{symmetric} if $\alpha_g = 2g-1$. Every semigroup generated by two elements is symmetric  \cite{rosales}. The curve $\mathcal{X}$ is called \emph{Castle curve} if $H(P)= \{ 0=m_1 < m_2 < \cdots \}$ is symmetric and $\# \mathcal{X}(\mathbb{F}_q)= m_2 q + 1$.
Details about Castle curves and their application in coding theory can be found in \cite{MST}.

\begin{definition}\label{telescopic}
Let $(a_{1}, \ldots , a_m)$ be a sequence of positive integers whose  greatest common divisor is $1$. Define $d_i := gcd (a_1, \ldots , a_i)$ and $A_i:= \{ \frac{a_1}{d_i}, \ldots , \frac{a_i}{d_i} \}$, for $i=1,\ldots,m$. Let $d_0=0$. Let $S_i$ be the semigroup generated by $A_i$. If $\frac{a_i}{d_i} \in S_{i-1}$, for $i=2, \ldots , m$, then the sequence $(a_1 , \ldots , a_m)$ is called \textit{telescopic}. We call a semigroup telescopic if
it is generated by a telescopic sequence.
\end{definition}

For a numerical semigroup $S$, the number of gaps and the largest gap of $S$ will be denoted by $g(S)$ and $l_g(S)$, respectively.

\begin{lemma} [\cite{kirfel}, Lemma 6.5] \label{lemma telescopic}
If $S_m$ is  the semigroup generated by a telescopic sequence $(a_1 , \ldots, a_m)$, then
\begin{itemize}

\item $\displaystyle l_g(S_m) = d_{m-1} l_g(S_{m-1}) + (d_{m-1}-1)a_{m} = \sum_{i=1}^{m} \left(\frac{d_{i-1}}{d_i}-1\right)a_i$

\item $g(S_m)= d_{m-1} g(S_{m-1}) + (d_{m-1}-1)(a_{m}-1)/2 = (l_{g}(S_m) + 1)/2,$

\end{itemize}
where  $d_{0}=0$. In particular, telescopic semigroups are symmetric.
\end{lemma}

\subsection{AG codes} Given a divisor $G$ on $\mathcal{X}$, for distinct rational points $P_{1}, P_{2}, \ldots , P_{n}$ on $\mathcal{X}$, with $P_{i} \notin \mbox{Supp}(G)$ for all $i$, consider the map
$$
ev_{\mathcal{L}}: \mathcal{L}(G) \rightarrow \mathbb{F}_{q}^{n} \mbox{ , } f \mapsto (f(P_{1}), f(P_{2}), \ldots , f(P_{n})).
$$

Let $D,G \in \Div(\mathcal{X})$ be such that $\mbox{Supp}(G) \cap \mbox{Supp}(D) = \emptyset$. The AG code, denoted by $C_{\mathcal{L}}(\mathcal{X},D,G)$, associated with the divisors $D$ and $G$, is defined as the image of the map $ev_{\mathcal{L}}$, that is, $C_{\mathcal{L}}(\mathcal{X},D,G):= ev_{\mathcal{L}}(\mathcal{L}(G))$.

Another code can be associated with the divisors $D$ and $G$,  by using the map $ev_{\Omega}: \Omega(G) \rightarrow F_{q}^{n} \mbox{ , } \omega \mapsto (res_{P_{1}}(\omega), res_{P_{2}}(\omega), \ldots , res_{P_{n}}(\omega))$. Define $C_{\Omega}(\mathcal{X},D,G):= ev_{\Omega}(\Omega(G-D))$. The code $C_{\Omega}(\mathcal{X},D,G)$  is  also called  AG, and  $C_{\mathcal{L}}(\mathcal{X},D,G)$ and $C_{\Omega}(\mathcal{X},D,G)$ are dual to each other. If $G=mQ$ for some rational point $Q$ on $\mathcal{X}$, where $m$ is an integer, then the codes $C_{\mathcal{L}}(\mathcal{X},D,mQ)$ and $C_{\Omega}(\mathcal{X},D,mQ)$ are called \textit{one-point AG codes}.

Let $[n,k,d]$, respectively,  be the length, dimension and minimum distance of an AG code. It is well known that determining  the parameter $d$ can be very challenging. Nevertheless, there are some  bounds in the literature, such as  the Singleton, Goppa and the order bound, from  which one  can estimate the minimum distance $d$. In this work, we will use a bound known as $d^*$. The bound $d^*$ was introduced by Andersen and Geil in \cite{andersen} and investigated by Geil, Munuera, Ruano, and Torres in \cite{GMRT}. The following is a brief introduction to  the bound $d^*$, which will be used  in the last section of this paper. For additional details, see \cite{GMRT}.

Consider the one-point AG code $C_{\mathcal{L}}(\mathcal{X},D,mQ)$, where $D=P_1 + \cdots + P_u$. Let $H(Q)= \{ 0=h_1 < h_2 < \cdots \}$ be the Weierstrass semigroup of $\mathcal{X}$ at $Q$. If $h_i < u=deg(D)$, then  $C_{\mathcal{L}}(\mathcal{X},D,h_iQ)$ has dimension  $k=i$. Consider the set
$$H^*=H^*(D,Q):=\{m\in \mathbb{N}_0: C_{\mathcal{L}}(\mathcal{X},D,mQ)\neq C_{\mathcal{L}}(\mathcal{X},D,(m-1)Q)\}\;.$$

Thus,  knowing the set $H^*$ is equivalent to knowing the dimension of all codes $C_{\mathcal{L}}(\mathcal{X},D, mQ)$. We observe that $H^*$ consists of $u$ elements and $H^* \subseteq H$. In addition, for $m < u$, $m \in H^*$ if and only if $m\in H$, see \cite{GMRT}. In particular, if $\mathcal{X}$ is a Castle curve, then $H^*=H(Q) \setminus (u+H(Q))$, see \cite[Example 3.8]{GMRT}.

Let $H^*=\{ 0=m_1 < m_2 < \cdots < m_u \}$. For $i=1,\dots,u$, let $\Lambda^*_i:=\{ m\in H^* : m-m_i \in H^* \}$ and define $d^*(i):=\min\{\#\Lambda^*_1,\dots,\#\Lambda^*_{i}\}$.

If $d$ is the minimum distance of $C_{\mathcal{L}}(\mathcal{X},D,m_iQ)$, then $d \geq d^*(i)$, see   \cite[Theorem 3.6]{GMRT}. Thus  the minimum distance $d$ can be  estimated by finding $d^*$.

For more details about coding theory, see \cite{vanlint}, \cite{stichtenoth2} and \cite{vanlint2}.

\subsection{The curve $\mathcal{X}_{n,r}$}

 Fix an integer $n\geq 2$. For any $r \in \{\lceil \frac{n}{2} \rceil,\ldots,n-1\}$, with $\gcd(n,r)=1$, define
\begin{equation}\label{eq:f_definition}
f_{r}(x):=T_n\left(x^{1+q^{r}}\right) \mod (x^{q^n}-x),
\end{equation}
 where  $T_n(x)=x+x^q+\ldots +x^{q^{n-1}}$ and consider the curve $\mathcal{X}_{n,r}$ defined over $\F_{q^n}$ with  affine equation given by
\begin{equation} \label{eq Xr curve}
T_n(y)=f_r(x).
\end{equation}
The curve $\mathcal{X}_{n,r}$ has degree $\delta=q^{n-1}+q^{r-1}$, genus $g(\mathcal{X}_{n,r})=q^r(q^{n-1}-1)/2$,  $\# \mathcal{X}_{n,r}(\mathbb{F}_{q^n})=q^{2n-1}+1$ $\F_{q^n}$-rational points, and a single rational point at infinity $Q_{\infty}=(0:1:0)$, see [\cite{c2}, Theorem 2.3]. Moreover, $\mathcal{X}_{n,r}$ is a Castle curve, see \cite[Corollary 3.7]{bst}.

By \cite[Proposition 3.5]{bst}, we have that

\begin{equation} \label{eq XrN curve}
T_n(y)=x^{q^n+q^{n-r}}-x^{q^{n-r}+1}
\end{equation}

is another plane model for the curve $\mathcal{X}_{n,r}$. This model will be used throughout the text.

It is easy to check that the point $P_\infty=(0:1:0)\in \mathcal{X}_{n,r}$, with plane model given by \eqref{eq XrN curve}, corresponds to the point $Q_\infty$ in the plane model given by \eqref{eq Xr curve} (see proof of \cite[Proposition 3.5]{bst}).
In particular, $H(P_\infty)=H(Q_\infty)$ and then \cite[Theorem 3.7]{bst} yields

%Let $z_0:=y^{q^{n-r}}-x^{q^{n-r}+1} \in \F_{q^n}(\mathcal{X}_{n,r})$. Let $\alpha$ and $\beta$ be positive integers such that $(n-r)\alpha-\beta n=1$, and consider the following functions in $\F_{q^n}(\mathcal{X}_{n,r})$.
%
%\begin{equation}\label{functions zwt}
%\begin{array}{l}
%z := z_0^{q^{2r-n}}-x^{q^r + 1} + x^{q^{2r-n} - 1}y.\\
%\\
%w:=\sum\limits_{i=0}^{\alpha-1}{z_0}^{q^{(n-r)i}}-\sum\limits_{i=1}^{\beta}{(x^{q^{n-r}+1}+x^{q^n+q^{n-r}})}^{q^{n(\beta-i)+1}}.\\
%\\
%t:=x^{q^{2r-n+1}-q}w+{z}^q+x^{q^{2r-n+1}-q^{2r-n}-q+1}z.
%\end{array}
%\end{equation}

\begin{theorem} \label{theorem H(P) da curva}
Let $H(P_{\infty})$ be the Weierstrass semigroup at $P_{\infty}\in \mathcal{X}_{n,r}$. Then
$$
H(P_{\infty}) = \langle q^{n-1}, q^{n-1}+q^{r-1}, q^{2r-1}+q^{n-r-1}, q^n+q^{n-r}, q^{2r}-q^n+q^r+1 \rangle.
$$
Moreover, $H(P_{\infty})$ is a telescopic semigroup and therefore symmetric.
\end{theorem}

Note that the previous Theorem and \cite[Lemma 3.2, Proposition 3.4]{bst} yield explicit bases for the Riemann-Roch spaces $\mathcal{L}(mP_\infty)$.

%Let $\omega_1,\ldots,\omega_5\in \mathbb{F}_{q^n}(\mathcal{X}_{n,r})$ be rational functions for which the pole number at $P_\infty$ are $q^{n-1}, q^{n-1}+q^{r-1}, q^{2r-1}+q^{n-r-1}, q^n+q^{n-r}, q^{2r}-q^n+q^r+1$, respectively.
%
%In the following  result, a  basis of $\mathcal{L}(mP_\infty)$ is described using the Weierstrass semigroup $H(P_\infty)$.
%
%\begin{proposition}
%Let $w,z,t$ be the functions as in (\ref{functions zwt}) and $m\in\mathbb{Z}$. Then, the set
%  $$
%\begin{array}{ll}
%\{x^iy^jw^{k_1}z^{k_2}t^{k_3} : & 0\leq i, 0\leq j<q^{n-r}, 0\leq k_1<q^{2r-n-1}, 0\leq k_2< q, 0\leq k_3<q^{n-r-1},\mbox{ and } \\
% & iq^{n-1}+j(q^{n-1}+q^{r-1})+k_1(q^{n}+q^{n-r})+k_2(q^{2r-1}+q^{n-r-1})+ \\ & k_3(q^{2r}-q^n+q^r+1)\le m\}\,,
%\end{array}
%  $$
%is a basis of $\mathcal{L}(mP_\infty)$.
%\end{proposition}
%
%\begin{proof}
%By Lemma 3.2 and Proposition 3.4 in \cite{bst}, we have that $\dv_\infty(x)=q^{n-1}P_\infty,$  $ \dv_\infty(y)=(q^{n-1}+q^{r-1})P_\infty,$   $\dv_\infty(w)=(q^n+q^{n-r})P_\infty,$   $\dv_\infty(z)=(q^{2r-1}+q^{n-r-1})P_\infty$ and $\dv_\infty(t)=(q^{2r}-q^n+q^r+1)P_\infty$, and the result follows.
%\end{proof}

\subsection{$q$-polynomials} A polynomial of the form $p(x)=a_{\ell} x^{q^{\ell}}+a_{\ell-1} x^{q^{\ell-1}}+\cdots+a_1 x^q + a_0x$ with coefficients in a field extension  $\mathbb{F}_{q^m}$ of $\mathbb{F}_q$ is called a $q$-polynomial over $\mathbb{F}_{q^m}$ (also known as linearized polynomial). Note that $a_0 \neq 0$ if and only if  $p(x)$ is separable. Moreover, if $F$ is an arbitrary field extension  of $\mathbb{F}_{q}$ and $p(x)$ is a $q$-polynomial, then
$$
\begin{array}{ll}
p(\alpha + \beta)=p(\alpha) + p(\beta), & \mbox{for all } \alpha, \beta \in F, \mbox{ and}\\
p(c \beta) = c \mbox{ } p(\beta), & \mbox{for all } c \in \mathbb{F}_q \mbox{ and all } \beta \in F.
\end{array}
$$

Therefore, if $F$ is any extension field of $\mathbb{F}_{q}$ in which $p(x)$ splits, then the set $R$ of roots of $p(x)$ in $F$ is an $\mathbb{F}_q$-vector space. In this case, if $p(x)=a_{\ell} x^{q^{\ell}}+a_{\ell-1} x^{q^{\ell-1}}+\cdots+a_1 x^q + a_0x$, with $a_{\ell}a_0\neq 0$, we have that $|R|=q^{\ell}$ and thus $R$ is an $\ell$-dimensional  $\mathbb{F}_q$-vector space.

\begin{theorem} \cite[Theorem 3.52]{lidl} \label{teo q_polynomial}
Let $R \subset \mathbb{F}_{q^m}$ be an $\mathbb{F}_q$-vector space. Then for any non-negative integer $k$ the polynomial
$$
\displaystyle p(x) = \prod_{\beta \in R} (x - \beta)^{q^k}
$$
is a $q$-polynomial over $\mathbb{F}_{q^m}$.
\end{theorem}

\begin{proposition} (Euclidean algorithm for $q$-polynomials) \label{algoritmo}
Let $f(x),g(x) \in \F_{q^m}[x]$ be two $q$-polynomials. Then there exist unique $q$-polynomials $Q(x),R(x) \in \F_{q^m}[x]$ such that $f(x)=Q(g(x))+R(x)$ and $0\leq \deg (R(x)) < \deg(g(x))$.
\end{proposition}

For further information regarding $q$-polynomials we refer the reader to \cite{lidl}.

Let $n$, $r$, and $T_n(x)$ be  given as previously. The following three results  will be important for the study of subcovers of $\mathcal{X}_{n,r}$ in the next section.

\begin{lemma}\label{quebra}
For any  positive integer $s<n-1$, there exist monic  $q$-polynomials $g(x)$ and $g_s(x)$ in $\F_{q^n}[x]$,  such that $g_s(g(x))=T_n(x)$, $\deg (g_s(x))=q^{s}$, and $g_s(x)$ splits into linear factors over $\F_{q^n}$.
\end{lemma}

\begin{proof}
Let $\mathcal{B}=\{\xi_1,\cdots,\xi_{n-1}\} \subset \F_{q^n}$ be a basis for the  $\F_q$-vector space given by the roots
of $T_n(x)=0$. By Theorem \ref{teo q_polynomial}, the subspace generated by any choice of $n-1-s$ elements in  $\mathcal{B}$ gives rise to monic $q$-polynomial $g(x)\in \F_{q^n}[x]$ of degree $q^{n-1-s}$ that divides $T_n(x)$. Thus the statement follows directly from the  Euclidean algorithm for $q$-polynomials.
\end{proof}

\begin{lemma}\label{LemaSqS} Let $g_s(y)= y^{q^s}+a_{s-1}y^{q^{s-1}}+\cdots+a_0y \in \F_{q^n}[y]$ be a separable  $q$-polynomial. Then there exists a polynomial $G(x,y)\in \F_{q^n}[x,y]$, such that
\begin{equation}\label{SqS}
G(x,y)^q-G(x,y)=x^{q^{n}}g_s(y)-yR_s(x)^{q^{n-s}},
\end{equation}
 where  $R_s(x)=(a_0x)^{q^s}+(a_1x)^{q^{s-1}}+\cdots+(a_{s-1} x)^{q}+x$.
 \end{lemma}
\begin{proof}
 Define recursively the following polynomials in  $\F_{q^n}[y]:$

(i) $G_1(y)=y$, and

(ii) $G_i(y)=G_{i-1}(y)^q+a_{s-i+1}^{q^{n-s+i-1}}y$, for $i=2,\ldots,s+1.$

Direct computation shows that $G_{s+1}(y)=g_s(y)$. Moreover,    for
$$G(x,y):= x^{q^{n-1}}G_{s}(y)+ x^{q^{n-2}}G_{s-1}(y)+\cdots +x^{q^{n-s}}G_{1}(y),$$ it follows that
$$G(x,y)^q-G(x,y)=x^{q^{n}}g_s(y)-yR_s(x)^{q^{n-s}},$$
where $R_s(x)=(a_0x)^{q^s}+(a_1x)^{q^{s-1}}+\cdots+(a_{s-1} x)^{q}+x$, which gives the result.
\end{proof}

\begin{proposition}\label{propchave} For $s\geq 1$,  let $g_s(y)=y^{q^s}+a_{s-1}y^{q^{s-1}}+\cdots+a_1y^{q}+a_0y\in \F_{q^n}[y]$ be a separable $q$-polynomial  that splits over  $\F_{q^n}$. Then there exist unique $q$-polynomials $Q(y), U(y)\in \F_{q^n}[y]$  such that

$$Q(y)^{q^{n-r}}=y+U(g_s(y)),$$
and   $\deg (U(y))=q^u\leq n-r-1$.  Moreover, if $s\geq r+1$, then  $Q(y)=y^{r}$ and
$u=n-s$.
\end{proposition}

\begin{proof}  By the Euclidean algorithm for $q$-polynomials, there exist unique $q$-polynomials
$Q(y) ,R(y) \in \F_{q^n}[y]$ such that $y^{q^{n-t}}=R(g_s(y))+Q(y)$ and $0<\deg (Q(y)) \leq q^{s-1}$. In particular,
$g_s(y)$ divides $y^{q^{r}}-Q(y)$ and $y^{q^{n}}-Q(y)^{q^{n-r}}$. Since $g_s(y)$ is a factor of $y^{q^n}-y$, it follows that  $g_s(y)$ divides  $Q(y)^{q^{n-r}}-y$. Now the  Euclidean algorithm gives a nonzero $q$-polynomial $U(y)\in \F_{q^n}[y]$ such that $Q(y)^{q^{n-r}}-y=U(g_s(y))$, and the result follows.
\end{proof}

\section{ Subcovers of  the curve $\mathcal{X}_{n,r}$}

For any positive integer $s<n-1$, let $g_s(y)$ be a polynomial of degree
$q^s$ given by Lemma \ref{quebra}, that is, $g_s(y)$ is a monic $q$-polynomial that splits into linear factors over $\F_{q^n}$ and $g_s(g(y))=T_n(y)$ for a monic $q$-polynomial $g(y)$ in $\F_{q^n}[y]$. Consider the curve $\mathcal{X}_{n,r}^s$ over $\mathbb{F}_{q^n}$ defined by
\begin{equation} \label{Maincurve}
\mathcal{X}_{n,r}^s: g_s(y)=x^{q^n+q^{n-r}}-x^{q^{n-r}+1}.
\end{equation}

The following result provides important information regarding curve $\mathcal{X}_{n,r}^s$.

\begin{theorem} \label{genus subcover} The curve $\mathcal{X}_{n,r}^s$  is a subcover of $\mathcal{X}_{n,r}$, and its number of  $\F_{q^n}$-rational points and genus are given by
$\# \mathcal{X}_{n,r}^s(\F_{q^n})=q^{n+s}+1  \text{ and } g(\mathcal{X}_{n,r}^s)=q^r(q^s-1)/2$,
respectively. Moreover,  $P_{\infty}=(0:1:0)$ is the only singular point in the projective closure of $\mathcal{X}_{n,r}^s$, and it corresponds to a unique point
nonsingular model of the curve.
\end{theorem}

\begin{proof}
Let $g(y)$ be the polynomial given by Lemma \ref{quebra}  such that $g_s(g(y))=T_n(y)$.
Note that the well-defined map $(x,y) \mapsto (x,g(y))$ shows that $\mathcal{X}_{n,r}^s$ is a subcover of $\mathcal{X}_{n,r}$.  Let $K(x)$ be the rational function field over $K=\bar{\F}_q$, and consider the field extension $E/K(x)$, where $E=K(x,y)$ and $x$ and $y$ satisfy \eqref{Maincurve}.   To compute the number of $\F_{q^n}$-rational points on $\mathcal{X}_{n,r}^s$, we first note that the pole of $x$ is totally ramified in $E/K(x)$ (see \cite[Theorem 3.3]{Deo}). That gives us one $\F_{q^n}$-rational point on $\mathcal{X}_{n,r}^s$, namely $P_{\infty}=(0:1:0)$. Lemma \ref{quebra} implies that  each  $\beta\in\F_{q^n}$
gives rise to $q^s$  $\F_{q^n}$-rational points on $\mathcal{X}_{n,r}^s$. Therefore
$\# \mathcal{X}_{n,r}^s(\F_{q^n})=q^{n+s}+1$.

To compute the genus, we rely on facts from  Artin-Schreier theory.
Let  $\wp:Y \mapsto Y^p-Y$  be the Artin-Schreier operator on $K(x)$ and  consider the set
$$
A:=\left\{\alpha^{q^n}x^{q^{r}+1}+\alpha^{q^r}x^{q^{n-r}+1}: R_s(\alpha)=0\right\} \subseteq K(x),
$$
where $R_s(x)$ is the separable polynomial given by Lemma \ref{LemaSqS}. Note that   $A \subseteq K[x]$   is an  additive subgroup of size  $|A|=q^{s}$, and
$$
A\cap \wp (K(x))=\{0\}.
$$
Therefore,  $K(x)(\wp^{-1}(A))$ is an elementary abelian $p$-extension of $K(x)$ of degree $|A|=q^{s}$ \cite[Section 1]{GH}. Now we  show that $K(x)(\wp^{-1}(A))=E$.

Choose $\alpha\in K^*$ such that  $R_s(\alpha)=0$, and  let  $G(x,y)$ be the polynomial given in Lemma \ref{LemaSqS}.  Consider   $z\in E$ given by
$$
z=G(\alpha,y)+T_{n-r}(\alpha^{q^{r}}x^{q^{r}+1}).
$$
Lemma \ref{LemaSqS} gives
\begin{eqnarray*}
z^q-z  &=& \alpha^{q^n}g_s(y)+(\alpha^{q^{r}}x^{q^{r}+1})^{q^{n-r}}-\alpha^{q^{r}}x^{q^{r}+1}\\
     &=& \alpha^{q^n}(x^{q^{n}+q^{n-r}}- x^{q^{n-r}+1})+(\alpha^{q^{r}}x^{q^{r}+1})^{q^{n-r}}-\alpha^{q^{r}}x^{q^{r}+1}\\
      &=&- \alpha^{q^n}x^{q^{n-r}+1}-\alpha^{q^{r}}x^{q^{r}+1} \in A.
\end{eqnarray*}
Also, for  $q=p^e$, we have that $z_t:=z^{p^{e-1}}+\cdots+z^p+z \in E$  satisfies
$$
z_t^p-z_t=z^q-z=- \alpha^{q^n}x^{q^{n-r}+1}-\alpha^{q^{r}}x^{q^{r}+1}.
$$
Therefore, $E$ contains the splitting fields of  $Y^p-Y-a$ for all $a \in A$, and \\
$[E:K(x)]=q^s$ gives  $K(x)(\wp^{-1}(A))=E$. Hence, by \cite[Theorem 2.1]{GH}, the genus of  $\mathcal{X}_{n,r}^s$ is
$$
g(\mathcal{X}_{n,r}^s)=\frac{q^{s}-1}{p-1}\cdot \frac{q^r(p-1)}{2}=\frac{q^r(q^{s}-1)}{2},
$$
which proves the assertion.
The final statement follows directly from the Jacobian criterion and  \cite[Theorem 3.3]{Deo}.
\end{proof}

\begin{lemma} \label{lema vp}
For the functions $x,y \in \F_{q^n}(\mathcal{X}_{n,r}^s)$, we have that
$$
(x)_{\infty}= q^sP_{\infty} \mbox{ and } (y)_{\infty}= (q^n+q^{n-r})P_{\infty}.
$$
\end{lemma}

\begin{proof}
From the proof of the Theorem \ref{genus subcover}, observing that $[K(x,y):K(x)]=q^s$ and $P_{\infty}$ is the unique pole of $x$, it follows that $(x)_{\infty}= q^sP_{\infty}$. Now, from (\ref{Maincurve}), it follows that $P_{\infty}$ is the unique pole of $y$. Thus  we have that $q^s v_{P_{\infty}}(y) = (q^{n}+q^{n-r} )v_{P_{\infty}}(x)$, that is, $v_{P_{\infty}}(y) = q^{n}+q^{n-r}$. Therefore, $(y)_{\infty}= (q^n+q^{n-r})P_{\infty}$.
\end{proof}

\begin{theorem}   \label{teo Z}
For $g_s(y)$ as in (\ref{Maincurve}), let  $U(y)$ and $Q(y)$ be given as in Proposition \ref{propchave}, with $\deg (U(y))=q^u$. If $h(x)\in \F_{q^n}[x]$ is  the unique polynomial  for which $U(x^{q^n+q^{n-r}})=h(x)^{q^{n-r}}$ and $Z:=Q(y)-h(x) \in \F_{q^n}(\mathcal{X}_{n,r}^s)$, then
\begin{equation}
(Z)_{\infty}= \begin{cases}
(q^{r}+1)P_{\infty},  \text{ if }\quad s\leq r-u-1,\\
q^{s+u+r-n}(q^{n-r}+1)P_{\infty},  \text{ if }\quad s\geq r-u.
\end{cases}
  \end{equation}

  In particular,
  \begin{equation}\label{particular}
(Z)_{\infty}= \begin{cases}
(q^{r}+1)P_{\infty},  \text{ if }\quad s\leq 2r-n ,\\
(q^{n}+q^r)P_{\infty}, \text{ if }\quad s\geq r+1.
\end{cases}
  \end{equation}
     \end{theorem}

\begin{proof}  From Proposition $\ref{propchave}$,  we have
$$Z^{q^{n-r}}=y+U(g(y))-U(x^{q^n+q^{n-r}})=y+U(x^{q^n+q^{n-r}}-x^{q^{n-r}+1})-U(x^{q^n+q^{n-r}})=y-U(x^{q^{n-r}+1}).$$
Note that $v_{P_{\infty}}(U(x^{q^{n-r}+1}))=q^u(q^{n-r}+1)v_{P_{\infty}}(x)=-q^{u+s}(q^{n-r}+1)$ and $v_{P_{\infty}}(y)=-q^n-q^{n-r}$. Thus, as $2r\neq n$, we have   $v_{P_{\infty}}(U(x^{q^{n-r}+1}))\neq v_{P_{\infty}}(y)$, and then
\begin{equation*}
q^{n-r}v_{P_{\infty}}(Z)= \begin{cases}
-(q^{r}+1)q^{n-r},  \text{ if }\quad s< r-u ,\\
-(q^{n-r}+1)q^{s+u} \text{ if }\quad s\geq  r-u\;,
\end{cases}
  \end{equation*}
 and the first assertion follows. The second assertion follows from the fact that $2r+1-n \leq r -u\leq r$,  and that $s\geq r+1$ implies $u=n-s$.

\end{proof}

 \begin{theorem}\label{Teo HP}
Let $H(P_{\infty})$ be the Weierstrass semigroup at $P_{\infty} \in \mathcal{X}_{n,r}^s$.
If  $s\leq 2r-n$ then
$$
H(P_{\infty}) = \langle q^{s}, q^{r}+1 \rangle.
$$

If  $s=2r-n+1$ then
\[
H(P_{\infty}) = \begin{cases}
 \langle q^{s}, q^{r}+q^{s-1}, q^{r+1}+q, q^{r+s-1}+1 \rangle,  \text{ if }\quad u=n-r-1 ,\\
\langle q^{s}, q^{r}+1 \rangle,  \text{ otherwise.}
\end{cases}
\]

Moreover, $H(P_{\infty})$ is a telescopic semigroup and therefore symmetric.
\end{theorem}

\begin{proof}

Let us  write
\begin{equation}\label{eqq1}
U(x)=U_1(x)^{q^{n-r-1}}+U_2(x),
\end{equation}
where $U_1(x)^{q^{n-r-1}}=a_ux^{q^{n-r-1}}$, and  $\deg U_2(x)\leq q^{n-r-2}$.  Defining $W:=Z^q+U_1(x^{q^{n-r}+1})$, where $Z$ is the function given in Theorem \ref{teo Z}, equation  \eqref{eqq1} together with $Z^{q^{n-r}}=y-U(x^{q^{n-r}+1})$ gives $W^{q^{n-r-1}}=y-U_2(x^{q^{n-r}+1})$, and thus
\begin{equation}
 \begin{cases}
v_P(U_2(x^{q^{n-r}+1}))\geq -(q^{n-1}+q^{r-1})\\
v_P(y)=-(q^{n}+q^{n-r})
\end{cases}
\end{equation}
implies $v_P(W)=-(q^{r+1}+q).$ Now let $\alpha, \beta \in \mathbb{F}_{q^n}^*$ be the unique elements for which
$-a_u^{-1}=\alpha=\beta^{q^{n-r}}$ and consider  the function
$$\Gamma:=\beta x^{q^{r-1}-q^{n-r-1}}\cdot Z -W^{q^{2r-n-1}}.$$
We have
\begin{eqnarray*}
\Gamma^{q^{n-r}} & = & \alpha x^{q^{n-1}-q^{2(n-r)-1}}\cdot Z^{q^{n-r}}-(W^{q^{n-r-1}})^{q^{2r-n}}\\
& = & \alpha x^{q^{n-1}-q^{2(n-r)-1}}\cdot (y-U(x^{q^{n-r}+1}))- (y-U_2(x^{q^{n-r}+1}))^{q^{2r-n}}\\
& = & ( y^{q^{2r-n}} - x^{q^{n-1}+q^{n-r-1}})+p(x)+(\alpha  y x^{q^{r+u}-q^{n-r+u}}),
\end{eqnarray*}
where $$p(x):=-\alpha x^{q^{n-1}-q^{2(n-r)-1}}\cdot U_2(x^{q^{n-r}+1})-U_2(x^{q^{n-r}+1})^{q^{2r-n}}.$$
One  can easily check that
\begin{equation}\label{ineqq}
 \begin{cases}
 v_P(y x^{q^{n-1}-q^{2(n-r)-1}})=-(q^{n-r}+q^{2r})\\
  v_P( p(x))=-q^{2r-n+1}\cdot \deg p(x) \geq -q^{2r-n+1}\\
 v_P( y^{q^{2r-n}} - x^{q^{n-1}+q^{n-r-1}})=-(q^{2r-1}+q^{r-1})\;,
\end{cases}
\end{equation}

where the latter case is obtained directly from the identity
$$x^{q^{n}+q^{n-r}}-y^{q^{2r-n+1}} =x^{q^{n-r}+1}+(y^{q^{2r-n}}+\cdots+y^q+y)\;.$$

Now triangle inequality and  \eqref{ineqq} give
$$v_P(\Gamma)=-(q^{3r-n}+1)=-(q^{r+s-1}+1),$$
and the result follows.
\end{proof}

\section{AG codes arising from $\mathcal{X}_{n,r}$ and their subcovers}

In this section, we will study one-point AG codes  arising from  the curve
\begin{equation}\label{newmodel}
\mathcal{X}_{n,r}\\:\\y^{q^{n-1}}+\cdots+y^q+y=x^{q^{n-r}+1}-x^{q^{n}+q^{n-r}}
\end{equation}
and its subcovers.
Let $P_\infty$ be the only rational  point corresponding to  $(0:1:0)$ in  the projective closure of $\mathcal{X}_{n,r}$, and let $P_1, \ldots , P_{q^{2n-1}}$ be the remaining  $q^{2n+1}$ $\mathbb{F}_{q^n}$-rational points on $\mathcal{X}_{n,r}$. Consider the sequence of one-point AG codes $(C_{m}({\mathcal{X}_{n,r}}))_{m\geq 1}$, where $C_{m}({\mathcal{X}_{n,r}}):=C_{\mathcal{L}}({\mathcal{X}_{n,r}},D, mP_\infty)$, with $D=\displaystyle\sum_{i=1}^{q^{2n-1}}P_i$. Let $k_{m,r}$ and  $d_{m,r}$ be the dimension and the minimum distance of $C_{m}({\mathcal{X}_{n,r}})$, respectively.

Again, let $H(P_\infty)=\{0= h_1<h_2<\ldots \}$ be the Weierstrass semigroup at $P_\infty$. Define the function $\iota=\iota_{P_\infty}: {\mathbb{N}}_0\rightarrow{\mathbb{N}}$ by $\iota(m)=\max\{ i :h_i\le m \}$. Then, we have that $\iota(m)=\ell (mP_\infty)$, where $\ell(mP_\infty)$ is the dimension of the vector space $\mathcal{L}(mP_\infty)$.

By Corollary 3.8 in \cite{bst}, $\mathcal{X}_{n,r}$ is a Castle curve and then we have the following result.

\begin{proposition}\cite[Proposition 5]{MST}\label{castle-codes}
Let $C_{m}({\mathcal{X}_{n,r}})$ as above. Then,
 \begin{enumerate}
\item  For $m<q^{2n-1}$, the dimension of $C_{m}({\mathcal{X}_{n,r}})$ is $k_{m,r}=\iota(m);$

\item  For $m\ge q^{2n-1},$ $C_{m}({\mathcal{X}_{n,r}})$ is an abundant code of abundance $\iota(m-q^{2n-1})$ and dimension
$k_{m,r}=\iota(m)-\iota(m-q^{2n-1});$

\item The dual of $C_{m}({\mathcal{X}_{n,r}})$ is isometric to $C_{q^{2n-1}+2g-2-m}({\mathcal{X}_{n,r}});$

\item For $1\le m<q^{2n-1},$ $d_{m,r}$ reaches Goppa bound if and only if $d_{q^{2n-1}-m,r}$ does$;$

\item  The minimum distance $d_{q^{2n-1},r}$ of $C_{q^{2n-1}}({\mathcal{X}_{n,r}})$ verifies $d_{q^{2n-1},r}\ge h_2.$
  \end{enumerate}
  \end{proposition}

The following proposition gives us the true minimum distance of the codes $C_{m}({\mathcal{X}_{n,r}})$ for all $m$, with $0\le m \leq q^{2n-1}$.

\begin{proposition}\label{distance-codes}
Let $d_{m,r}$ be the minimum distance of $C_{m}({\mathcal{X}_{n,r}})$. Then,
\begin{enumerate}
	\item If	$m=aq^{n-1}$ with 	$0\leq a<q^n$, then $d_{m,r}=q^{2n-1}-m$.
	\item If $m=aq^{n-1}+b(q^{n-1}+q^{r-1})$ with $0\leq a\leq q^n-q^{n-1}-q^{r-1}$ and $0\leq b< q^{n-1}$, then $d_{m,r}=q^{2n-1}-m$.
	\item If $m=q^{2n-1}-q^{n-1}+b$ with $0\leq b\leq q^{n-1}$, then $d_{m,r}=q^{n-1}$.
\end{enumerate}
\end{proposition}

\begin{proof}
\begin{enumerate}
	\item Let $\alpha_1,\ldots,\alpha_a$ be distinct elements in $\mathbb{F}_{q^n}$. Since $(x)_\infty=q^{n-1}P_\infty$, we have that $g=(x-\alpha_1)\cdots (x-\alpha_a)\in \mathcal{L}(mP_\infty)$, and for  each $\alpha\in\mathbb{F}_{q^n}$ the line $x=\alpha$ intersects the affine curve $\mathcal{X}_{n,r}$ at $q^{n-1}$ distinct rational points. Therefore, the codeword $ev_{\mathcal{L}}(g)\in C_{m}({\mathcal{X}_{n,r}})$ has weight $q^{2n-1}-m$.

	\item Fix $\gamma\in\mathbb{F}_q^*$ and let $A=\{\alpha\in\mathbb{F}_{q^n}: \alpha^{q^{n-r}+1}-\alpha^{q^n+q^{n-r}}\neq \gamma\}$. Then $\#A\geq q^n-q^{n-1}-q^{r-1}\geq  a$. Choose distinct elements $\alpha_1,\hdots,\alpha_a\in A$ and define
		$$
	g_1:=\prod_{\mu=1}^a(x-\alpha_\mu).
		$$
Note that $g_1$ has $aq^{n-1}$ distinct zeros in the support of $D$. On the other hand, there exist $q^{n-1}$ distinct elements $\beta\in \mathbb{F}_{q^n}$ such that $\beta^{q^{n-1}}+\beta^{q^{n-2}}+\cdots+\beta=\gamma$. Choose $b$ of them and define a function
		$$
	g_2:=\prod_{\nu=1}^b(y-\beta_{\nu}).
		$$
Therefore, $g_2$ has $b(q^{n-1}+q^{r-1})$ distinct zeros in the support of $D$, all of them being distinct from the zeros of $g_1$. Then $g_1g_2\in \mathcal{L}(mP_\infty)$ has $m$ distinct zeros in the support of $D$, and the corresponding codeword $ev(g_1g_2)$ has weight $q^{2n-1}-m$.
	
	\item From (1), we have $d_{q^{2n-1}-q^{n-1},r}=q^{n-1}$. Therefore $d_{m,r}=q^{r-1}$ by Proposition \ref{castle-codes}.
\end{enumerate}
\end{proof}

\begin{example}
Let $q=2$, $n=4$, and $r=3$. Then, we have the curve $\mathcal{X}_{4,3}: y^8+y^4+y^2+y=x^3-x^{18}$, with $g=28,\#\mathcal{X}_{4,3}=129$ and $H(P_\infty)=\langle 8,12,18,33 \rangle$.

By Proposition \ref{distance-codes}, we can also obtain the true minimum distance of the codes $C_{16}(\mathcal{X}_{4,3})$, $C_{20}(\mathcal{X}_{4,3})$ and $C_{24}(\mathcal{X}_{4,3})$. These codes have, respectively, the fo\-llowing parameters: $[128,4, 112], [128,6, 108],[128,7,104]$, and such parameters have achieved the best known parameters according to \cite{MinT}.

\end{example}

In the following, we present examples of AG codes whose parameters $[n,k,d]$ are better than those appearing in the MinT tables \cite{MinT}.

\begin{example}\label{Ex1}
Let $q=2,n=5$, and $r=3$. Then, we have the curve $\mathcal{X}_{5,3}: y^{16}+y^8+y^4+y^2+y=x^5-x^{36}$ over $\mathbb{F}_{32}$, with $g=60$, $\# \mathcal{X}_{5,3}=513$ and $H(P_\infty)=\langle 16,20,34,41 \rangle$.

Suppose that  $\mathbb{F}_{32}^{*} = \langle a \rangle$. Let $g_s(y)=y^4+a^{18}y^2+ay$ and $g(y)=y^4+a^8y^2+a^{30}y$ be in $\mathbb{F}_{32}[y]$. Then $g_s(g(y))=T_5(y)$ and, by Theorem \ref{genus subcover}, we have that 
$$\mathcal{X}_{5,3}^2:y^4+a^{18}y^2+ay=x^5-x^{36}\;$$
is a subcover of $\mathcal{X}_{5,3}$. Theorem \ref{genus subcover} gives $g(\mathcal{X}_{5,3}^2)=12$ and $\#\mathcal{X}_{5,3}^2(\mathbb{F}_{32})=129$, and Theorem \ref{Teo HP} gives $H(P_\infty)=\langle 4,10,17\rangle$.

%$\cdot$ $\mathcal{X}_{5,3}^3: a^{28}y^8 + a^{27}y^4 + a^{15}y^2 + a^{25}y=x^5-x^{36}$, with $g_{5,3}^3=28, \#\mathcal{X}_{5,3}^3=257$.

Therefore $H^*_2=\{ 0, 4, 8, 10, 12, 14, 16, 17,$ $18, 20, 21,$ $22, 24, 25,$ $26, 27, 28,$ $29, 30, 31,$ $32, 33, 34, 35, 36, 37$, $38, 39, 40,$ $41, 42, 43,$ $44, 45, 46,$ $47, 48, 49,$ $50, 51, 52,$ $53, 54, 55,$ $56, 57, 58,$ $59, 60, 61, 62, 63$, $64, 65, 66,$ $67, 68, 69,$ $70, 71,$ $72, 73, 74, 75,$ $76, 77, 78, 79,$ $80, 81, 82, 83,$ $84, 85, 86,$ $87, 88, 89$, $90, 91, 92,$ $93, 94, 95,$ $96, 97,$ $98, 99, 100,$ $101,$ $102,$ $103,$ $104, 105, 106,$ $107, 108, 109,$ $110, 111$, $112, 113,$ $114, 115, 116, 117,$ $118, 119,$ $120,$ $121,$ $122, 123, 124,$ $125, 126, 127,$ $129, 130, 131$, $133, 134, 135,$ $137, 139, 141,$ $143, 147,$ $151\}$.

Applying bound $d^*$ for the codes $C_{105}(\mathcal{X}_{5,3}^2)$ and $C_{109}(\mathcal{X}_{5,3}^2)$, we see that their res\-pective parameters $[128,94,24]$ and $[128,98,20]$ are better than those appearing in the MinT tables \cite{MinT}.
\end{example}

\begin{remark}
If there exists a linear code over $\mathbb{F}_{q}$ with parameters $[n,k,d]$, then for each non-negative integer $s < k$, there exists a linear code over $\mathbb{F}_{q}$ with parameters $[n-s,k-s,d]$, see \cite[Exercise 7, (iii)]{tsfasman}. From this fact, we see that the codes from the previous example give rise to other 14 codes whose parameters are better than those appearing in the MinT tables \cite{MinT}. In fact, using  $s=1,\ldots,7$ for the codes $C_{105}(\mathcal{X}_{5,3}^2)$ and $C_{109}(\mathcal{X}_{5,3}^2)$.
\end{remark}

\begin{example}\label{Ex 2} For the same particular values used in Example \ref{Ex1}, suppose again that $\mathbb{F}_{32}^{*} = \langle a \rangle$. Let $g_s(y)=y^8+a^{12} y^4+a^{20}y^2+ay$ and $g(y)=y^2+a^{30}y$ be in $\mathbb{F}_{32}[y]$. Thus $g_s(g(y))=T_5(y)$ and, by Theorem \ref{genus subcover}, we have that  
$$ \mathcal{X}_{5,3}^3: y^8+a^{12} y^4+a^{20}y^2+ay=x^5-x^{36}\;$$ is a subcover of $\mathcal{X}_{5,3}$. It follows that $g(\mathcal{X}_{5,3}^3)=28$ and $\#\mathcal{X}_{5,3}^3(\mathbb{F}_{32})=257$. We have that $(x)_\infty=8P_\infty$ and $(y)_\infty=20P_\infty$. Using \textit{MAGMA Computational Algebra System} \cite{magma} we find that $(y^4 + y^2 + y + a^{17}x^{10} + a^{17}x^9)_\infty=18P_\infty$ and $((a^7x^2 + a^{14}x + 1)y^4 + (a^7x^2 + a^{14}x)y^2 + (a^9x^2 + a^{14}x)y + a^{24}x^12 + a^{15}x^{11} + a^{30}x^{10} + a^{17}x^9 + a^{24}x^5)_\infty=25P_\infty$, and then $H(P_\infty)=\langle 8,18,20,25 \rangle$. 

Therefore $H^* = \{  0, 8, 16, 18, 20, 24, 25, 26, 28, 32, 33, 34, 36, 38,$ $40, 41,$ $42, 43,$ $44, 45,$ $46, 48,$ $ 49, 50,$ $51, 52,$ $53, 54,$ $56, 57,$ $58, 59,$ $60, 61,$ $62, 63,$ $64, 65,$ $66, 67,$ $68,$ $69, 70,$ $71, 72,$ $73, 74,$ $75,$ $ 76, 77,$ $78, 79,$ $80, 81,$ $82, 83,$ $84, 85,$ $86, 87,$ $88, 89,$ $90, 91,$ $92, 93,$ $94, 95, 96, 97, 98, 99, 100,$ $ 101, 102, 103, 104, 105, 106, 107, 108, 109, 110, 111, 112,$ $113,$ $114,$ $115,$ $116,$ $117, 118, 119,$ $ 120, 121, 122, 123, 124, 125, 126, 127, 128, 129, 130,$ $131,$ $132,$ $133,$ $134, 135,$ $136, 137, 138,$ $ 139, 140, 141,$ $142, 143, 144,$ $145, 146, 147,$ $148,$ $149,$ $150,$ $151, 152,$ $153,$ $154, 155, 156, 157,$ $ 158, 159, 160,$ $161, 162, 163,$ $164, 165,$ $166,$ $167,$ $168,$ $169,$ $170, 171,$ $172,$ $173, 174,$ $175, 176,$ $ 177, 178, 179,$ $180, 181, 182,$ $183, 184,$ $185,$ $186,$ $187, 188, 189, 190,$ $191, 192, 193, 194, 195,$ $ 196, 197, 198,$ $199, 200, 201,$ $202,$ $203,$ $204,$ $205,$ $206,$ $207,$ $208, 209,$ $210,$ $211, 212, 213, 214,$ $ 215, 216,$ $217, 218,$ $219, 220,$ $221, 222,$ $223, 224, 225,$ $226, 227, 228,$ $229, 230,$ $231, 232, 233,$ $ 234, 235,$ $236,$ $237,$ $238,$ $239,$ $240,$ $241,$ $242, 243,$ $244, 245,$ $246, 247,$ $248, 249,$ $250, 251, 252,$ $ 253, 254,$ $255,$ $257,$ $258,$ $259,$ $260,$ $261, 262, 263,$ $265, 266, 267,$ $268, 269, 270,$ $271, 273, 275,$ $ 277,$ $278,$ $279,$ $283,$ $285,$ $286,$ $287, 291, 293, 295, 303, 311 \}$.

Applying the bound $d^*$ for the codes $C_{201}(\mathcal{X}_{5,3}^3)$, $C_{209}(\mathcal{X}_{5,3}^3)$, $C_{217}(\mathcal{X}_{5,3}^3)$,  and $C_{219}(\mathcal{X}_{5,3}^3)$ we see that their respective parameters $[256,174,56]$, $[256,182,48]$, $[256,190,40]$, and $[256,192,38]$ are better than those appearing in the MinT tables \cite{MinT}. As before, using $s=1,\ldots,24$ for the codes $C_{201}(\mathcal{X}_{5,3}^3)$, $C_{209}(\mathcal{X}_{5,3}^3)$ and $C_{217}(\mathcal{X}_{5,3}^3)$ and using $s=1,\ldots,16$ for the code $C_{219}(\mathcal{X}_{5,3}^3)$, another 88 codes whose parameter are better than those appearing in the MinT tables \cite{MinT} can be constructed.
\end{example}

\begin{remark}
The Weierstrass semigroup at $P_\infty\in \mathcal{X}_{5,3}^3$ is not directly given by Theorem \ref{Teo HP}. Note that Example \ref{Ex 2} indicates that is worth investigating the  Weierstrass semigroup at $P_\infty\in \mathcal{X}_{n,r}^s$ for the case $s>2r-n+1$.
\end{remark}

\section{Aknowlegments}

The first author was supported by FAPESP (Brazil), grant 2017/04681-3.


\begin{thebibliography}{99}

\bibitem{andersen} H. Andersen and O. Geil, {\em Evaluation codes from order domain theory}, Finite Fields and Their Applications, vol. 14, pp. 92-123, 2008.

\bibitem{c2} H. Borges and R. Concei\c c\~ao,  {\em A new family of Castle and Frobenius nonclassical curves}, Journal of Pure and Applied Algebra, vol. 222, no. 4, pp. 994--1002, 2018.

\bibitem{bst} H. Borges, A. Sepúlveda and G. Tizziotti,  {\em Weierstrass semigroup and Automorphism group of the $\mathcal{X}_{n,r}$ curves}, Finite Fields and Their Applications, vol. 36, pp. 121-132, 2015.

%\bibitem{cossidente} A. Cossidente, G. Korchmáros and F. Torres, {\em On curves covered by the Hermitian curve}, Journal of Algebra, 216, 1999, 56-76.

%\bibitem{magma} ``Computer Algebra Group, University of Sydney", Magma Computational Algebra System, http://magma.maths.usyd.edu.au/magma, as viewed in January, 2018.

\bibitem{magma} W. Bosma, J. Cannon and C. Playoust, {\em The Magma algebra system I: the user language}, J. Symb. Computation, vol. 24, pp. 235-265, 1997.

\bibitem{Deo} V. Deolalikar, {\em Determining irreducibility and ramification groups for an additive extension of the rational function Field}, J. Number Theory, vol. 97, 2002, pp. 269--286.

\bibitem{fanali} S. Fanali and M. Giulietti, {\em One-point AG Codes on the GK Maximal Curves}, IEEE Trans. on Information Theory, vol. 56, no. 1, pp. 202 - 210, 2010.

\bibitem{GH} A. Garcia and H. Stichtenoth,{ \em Elementary abelian p-extensions of algebraic function fields}, Manuscripta Math. vol. 72, 1991, pp. 67--79.

%\bibitem{garcia} A. Garcia and H. Stichtenoth, {\em A maximal curve which is not a Galois subcover of the Hermitian curve}, Bulletin Brazilian Math. Society (N.S.), 37, 2006, 139--152.

%\bibitem{garcia} A. Garcia, S. J. Kim, and R. F. Lax, {\em Consecutive Weierstrass gaps
%and minimum distance of Goppa codes}, Journal of Pure and  Applied Algebra 84 (1993), 199--207.

\bibitem{rosales}
     \newblock P. Gárcia-Sánches and J. C. Rosales,
     \newblock \emph{Numerical Semigroups},
     \newblock Springer, Serie: Developments in Mathematics, vol. 20, 2009.

     \bibitem{GMRT} O. Geil, C. Munuera, D. Ruano and F. Torres,
{\em On the order bound for one-point codes}, Advances in Mathematics of Communication, vol. 5, no. 5, pp. 489-504, 2011.

%\bibitem{giulietti} M. Giulietti, L. Quoos and G. Zini, {\em Maximal curves from subcovers of the GK-curve}, Journal of Pure and Applied Algebra, 220 (10), 2016, 3372-3383.

\bibitem{Goppa1} V.D. Goppa, {\em Codes on Algebraic Curves}, Dokl. Akad. Nauk, SSSR, vol. 259, no. 6, pp. 1289-1290, 1981.

\bibitem{Goppa2} V.D. Goppa, {\em Algebraic Geometric Codes}, Izv. Akad. Nauk, SSSR, vol. 46, no. 4, pp. 75-91, 1982.

\bibitem{vanlint}
     \newblock T. H{\o}holdt, J. van Lint and R. Pellikaan,
     \newblock \emph{Algebraic geometry codes}, Elsevier, Handbook of Coding
Theory, \textbf{1}, Amsterdam, 1998.


\bibitem{kirfel} C. Kirfel and R. Pellikaan, {\em The minimum distance of codes in an array coming from telescopic
semigroups}, IEEE Trans. Inform. Theory, vol. 41, pp. 1720-1732, 1995.

\bibitem{lidl} R. Lidl and H. Niederreiter, {\em Introduction to Finite Fields and their Applications}, Cambridge Press, 1986.

\bibitem{gretchen} G. L. Matthews, {\em Weierstrass pairs and minimum distance of Goppa codes}, Designs, Codes and Cryptography, vol. 22, pp. 107-121, 2001.

\bibitem{gretchen2} G. L. Matthews, {\em Codes from the Suzuki function field}, IEEE Trans. on Information Theory, vol. 50, no. 12, pp. 3298-3302, 2004.

\bibitem{MinT}
MinT, {\em Online database for optimal parameters of $(t,m,s)$-nets, $(t,s)$-sequences, orthogonal arrays, and linear codes.} Online available at http://mint.sbg.ac.at.

\bibitem{MST} C. Munuera, A. Sepúlveda and F.Torres,
{\em Algebraic geometry codes from castle curves}, In: Barbero A. (ed.) Coding Theory and Applications, Springer, Heidelberg, pp. 117-127, 2008.

\bibitem{munuera} C. Munuera, G. Tizziotti and F. Torres, {\em Two-Points Codes on Norm-Trace Curves}, In: Barbero A. (ed.) Coding Theory and Applications, Springer, Heidelberg, pp. 128-136, 2008.

\bibitem{ST} A. Sep\'ulveda and G. Tizziotti, {\em Weierstrass Semigroup and codes over the curve $y^q+y=x^{q^r+1}$}, Advances in Mathematics of Communications, vol. 8, no. 1, pp. 67-72, 2014.

\bibitem{stichtenoth2} H. Stichtenoth, {\em Algebraic Function Fields and Codes}, Berlin, Germany: Springer, 1993

\bibitem{vanlint2} J. H. van Lint, {\em Introduction to Coding Theory}, New York: Springer 1982.

%\bibitem{tafazolian} S. Tafazolian, A. Teheran-Herrera and F. Torres, {\em Further examples of maximal curves which cannot be covered by the Hermitian
%curve}, J. Pure Appl. Algebra 220 (3), 2016, 1122-1132.

\bibitem{tsfasman} M. A. Tsfasman and S. G. Vladut, {\em Algebraic-Geometric Codes}, Amsterdam, The Netherlands: Kluwer, 1991.




\end{thebibliography}
\end{document}